\documentclass{amsart}%
\usepackage{amssymb}
\usepackage{amsmath}
\usepackage{amsfonts}
\usepackage{graphicx}%
\newtheorem{theorem}{Theorem}
\theoremstyle{remark}
\newtheorem*{remark}{Remark}
\theoremstyle{plain}
\newtheorem*{acknowledgement}{Acknowledgement}

\newtheorem{corollary}[theorem]{Corollary}

\newtheorem{lemma}[theorem]{Lemma}

\newtheorem{proposition}[theorem]{Proposition}

\numberwithin{equation}{section}
\begin{document}
\title[Explicit excision]{Explicit Wodzicki excision in cyclic homology}
\author{O. Braunling}
\address{Fakult\"{a}t f\"{u}r Mathematik, Universit\"{a}t Duisburg-Essen,
Thea-Leymann-Strasse 9, 45127 Essen, Germany}
\email{oliver.braeunling@uni-due.de}
\thanks{This work has been partially supported by the DFG SFB/TR45 \textquotedblleft%
{Periods, moduli spaces, and arithmetic of algebraic varieties}%
\textquotedblright\ and the Alexander von Humboldt Stiftung.}
\date{{\today}}

\begin{abstract}
Assuming local one-sided units exist, I give an elementary proof of Wodzicki
excision for cyclic homology. The proof is also constructive and provides an
explicit inverse excision map. As far as I know, the latter is new.

\end{abstract}
\maketitle

We work over a field of characteristic zero. For every algebra extension
$I\hookrightarrow A\twoheadrightarrow A/I$, where $I$ is a two-sided ideal,
there is a long exact sequence in cyclic homology%
\begin{equation}
\cdots\rightarrow HC_{n}(A,I)\rightarrow HC_{n}(A)\rightarrow HC_{n}%
(A/I)\rightarrow HC_{n-1}(A,I)\rightarrow\cdots\text{,} \label{lmj1}%
\end{equation}
where $HC_{n}(A,I)$ is relative cyclic homology. In general this group really
depends on $A$, but in favourable situations it agrees with the group
$HC_{n}(I)$. M. Wodzicki proved that this happens iff the bar complex of $I$
is acyclic \cite[Thm. 3]{MR934604}, \cite[\S 3]{MR997314}, $I$ is then called
$H$\emph{-unital}. In this case the straightforward \emph{excision map}, i.e.%
\begin{equation}
\rho:HC_{n}(I)\longrightarrow HC_{n}(A,I)\text{,} \label{lpj20}%
\end{equation}
is an isomorphism. In general it does not seem realistic to hope for a closed
formula for its inverse, already by the abstract nature of $H$-unitality.
Chances should get better if the bar complex comes with an explicit
contracting homotopy. A prominent such case is the following: An algebra $I$
\emph{has local left units} if for every finite set $S\subseteq I $ there
exists an element $e$ such that $\forall s\in S:es=s$. Wodzicki shows that
such $I$ are $H$-unital \cite[Prop. 2]{MR934604}, \cite[Cor. 4.5]{MR997314}.

In the present paper we want to give an elementary proof for this special case
of Wodzicki's theorem:

\begin{theorem}
\label{ANNOUNCER_ThmExcision}Suppose $I\hookrightarrow A\twoheadrightarrow
A/I$ is an algebra extension such that $I$ and $A$ have local left (or right)
units. Then the excision map $\rho:HC_{n}(I)\rightarrow HC_{n}(A,I)$ is an isomorphism.
\end{theorem}

The proof circumvents the use of spectral sequences, but on the downside is of
course not as general as the original result. In fact, the proof is
constructive and leads to an explicit inverse map $\rho^{-1}$.

\begin{theorem}
\label{ANNOUNCER_ThmExplicitFormula}We keep the assumptions as in Thm.
\ref{ANNOUNCER_ThmExcision}. Then for every finite-dimensional subspace of
$V\subseteq HC_{n}(A,I)$ there exists a finite-dimensional subspace
$V^{\prime}\subseteq I\otimes A^{\otimes n}$ and (non-canonical) elements
$e_{1},\dots,e_{n}\in I$ allowing to define a map $V^{\prime}\rightarrow
I^{\otimes n+1}$ sending $f_{0}\otimes f_{1}\otimes\cdots\otimes f_{n}$ to%
\begin{equation}
\sum_{s_{1}\ldots s_{n}\in\{\pm\}}\left(  -1\right)  ^{s_{1}+\cdots+s_{n}%
}\underset{s_{1}}{\underbrace{%
\begin{array}
[c]{c}%
e_{1}\otimes f_{1}\\
f_{1}e_{1}\otimes
\end{array}
}}\underset{s_{2}}{\underbrace{%
\begin{array}
[c]{c}%
e_{2}\otimes f_{2}\\
f_{2}e_{2}\otimes
\end{array}
}}\cdots\underset{s_{n}}{\underbrace{%
\begin{array}
[c]{c}%
e_{n}\otimes f_{n}\\
f_{n}e_{n}\otimes
\end{array}
}}f_{0} \label{lpwn2}%
\end{equation}
(where for each underbrace we take the upper term if $s_{i}=+$, the lower if
$s_{i}=-$) which in homology induces the inverse map $\rho^{-1}:HC_{n}%
(A,I)\supseteq V\rightarrow HC_{n}(I)$.
\end{theorem}

As $V$ can be picked large enough to contain any finite set of elements, this
describes $\rho^{-1}$ entirely. In low-degree cases the formula unwinds as%
\[
f_{0}\otimes f_{1}\mapsto e_{1}\otimes f_{1}f_{0}-f_{1}e_{1}\otimes f_{0}%
\]
for $n=1$; and for $n=2$ one gets%
\begin{align*}
f_{0}\otimes f_{1}\otimes f_{2}  &  \mapsto e_{1}\otimes f_{1}e_{2}\otimes
f_{2}f_{0}-f_{1}e_{1}\otimes e_{2}\otimes f_{2}f_{0}\\
&  -e_{1}\otimes f_{1}f_{2}e_{2}\otimes f_{0}+f_{1}e_{1}\otimes f_{2}%
e_{2}\otimes f_{0}\text{.}%
\end{align*}
We finish by spelling out the case $n=3$, which is already fairly involved:
$f_{0}\otimes f_{1}\otimes f_{2}\otimes f_{3}$ maps to%
\begin{align*}
&  +e_{1}\otimes f_{1}e_{2}\otimes f_{2}e_{3}\otimes f_{3}f_{0}-e_{1}\otimes
f_{1}e_{2}\otimes f_{2}f_{3}e_{3}\otimes f_{0}\\
&  -e_{1}\otimes f_{1}f_{2}e_{2}\otimes e_{3}\otimes f_{3}f_{0}+e_{1}\otimes
f_{1}f_{2}e_{2}\otimes f_{3}e_{3}\otimes f_{0}\\
&  -f_{1}e_{1}\otimes e_{2}\otimes f_{2}e_{3}\otimes f_{3}f_{0}+f_{1}%
e_{1}\otimes e_{2}\otimes f_{2}f_{3}e_{3}\otimes f_{0}\\
&  +f_{1}e_{1}\otimes f_{2}e_{2}\otimes e_{3}\otimes f_{3}f_{0}-f_{1}%
e_{1}\otimes f_{2}e_{2}\otimes f_{3}e_{3}\otimes f_{0}\text{.}%
\end{align*}
The unpleasant restriction to finite-dimensional subspaces in the theorem is
of technical nature. As we enlarge such a space, e.g. by taking the union of
two such subspaces, it need not be possible to choose the elements $e_{i}$ so
that the maps as in eq. \ref{lpwn2} prolong compatibly. Only after going to
homology, they all describe the same map $\rho^{-1}$.

This artifact comes from the fact that we can always pick local units for sets
$S$, but have no control how they behave as $S$ enlarges.

\begin{theorem}
\label{ANNOUNCER_ThmExplicitFormula2}We keep the assumptions as in Thm.
\ref{ANNOUNCER_ThmExplicitFormula}. Then every class $[\varphi]\in
HC_{n}(A,I)$ is represented by a cycle%
\[
\varphi=\sum\lambda_{j}\varphi_{j}\quad\text{with}\quad\varphi_{j}%
=f_{0}\otimes f_{1}\otimes\cdots\otimes f_{n}\in I\otimes A^{\otimes n}%
\]
i.e. a linear combination of pure tensors $\varphi_{j}$ with initial slot in
$I$. Take $\varphi^{1},\ldots,\varphi^{d}$ any elements whose classes are
spanning the finite-dimensional subspace $V$ in Thm.
\ref{ANNOUNCER_ThmExplicitFormula} and write $\varphi_{j}^{\alpha}\in I\otimes
A^{\otimes n}$ for the respective pure tensor components. Now pick $V^{\prime
}:=\operatorname*{span}\{\varphi_{j}^{\alpha}\}_{\alpha,j}$,

\begin{itemize}
\item $e_{n}\in I$ as a local left unit for $\bigcup_{\{\varphi_{j}^{\alpha
}\}}\{f_{0}\}$, where the union runs over the $f_{0}$-slots of all
$\varphi_{j}^{\alpha}$; $-$ and inductively going downward for each $2\leq
i\leq n$:

\item $e_{i-1}\in I$ as a local left unit for $\{e_{i}\}\cup\bigcup
_{\{\varphi_{j}^{\alpha}\}}\{f_{i}e_{i}\}$, where the union runs over the
$f_{i}$-slots of all $\varphi_{j}^{\alpha}$.
\end{itemize}

This provides a concrete choice of $V^{\prime}$ and the $e_{i}$ in the
statement of Thm. \ref{ANNOUNCER_ThmExplicitFormula}.
\end{theorem}

We also obtain some (weaker) results for Hochschild homology, see Prop.
\ref{marker_HHFormula}.

\section{Preparations\label{marker_PreparationsAndDefOfHH}}

Let $k$ be a field of characteristic zero. In this text, the word
\emph{algebra} refers to an associative $k$-algebra which need not be
commutative and especially not unital. Even if units exist, algebra morphisms
are not required to preserve them. Tensor products are always over $k$.

Let $A$ be an algebra. Write $C_{i}(A):=A^{\otimes i+1}$ for the underlying
groups of Hochschild homology; equip them with the usual differential%
\begin{align*}
b(f_{0}\otimes\cdots\otimes f_{n}):=  &  \sum\nolimits_{i=0}^{n-1}\left(
-1\right)  ^{i}f_{0}\otimes\cdots\otimes f_{i}f_{i+1}\otimes\cdots\otimes
f_{n}\\
&  +(-1)^{n}f_{n}f_{0}\otimes f_{1}\otimes\cdots\otimes f_{n}%
\end{align*}
so that $HH_{i}(A)=H_{i}(\{C_{\bullet}(A),b\})$ is the (na\"{\i}ve) Hochschild
homology of $A$. We shall also need the cyclic permutation operator%
\begin{equation}
t(f_{0}\otimes\cdots\otimes f_{n}):=(-1)^{n}\,f_{n}\otimes f_{0}\otimes
f_{1}\otimes\cdots\otimes f_{n-1}\text{.} \label{lpp3}%
\end{equation}
Define $CC_{n}(A):=C_{n}(A)_{\left\langle t\right\rangle }$, the co-invariants
under $t$, so that (na\"{\i}ve) cyclic homology is given by $HC_{i}%
(A)=H_{i}(\{CC_{\bullet}(A),b\})$.

\begin{remark}
[this definition suffices]This is the correct definition of Hochschild
homology only if $A$ is unital; in general one defines `correct Hochschild
homology' $HH_{i}^{\mathsf{corr}}(A)$ as the homology of the two-row bicomplex%
\[
C_{i}^{\mathsf{corr}}(A):=[C_{i}(A)\overset{1-t}{\longrightarrow}%
B_{i+1}(A)]\text{,}%
\]
where $B_{\bullet}$ is the bar complex and $t$ the cyclic permutation
operator; all details can be found in \cite[\S 2, esp. p. 598 l. 5]{MR997314}.
Equivalently, one can define $HH_{i}^{\mathsf{corr}}(A):=\operatorname*{coker}%
(C_{i}(k)\rightarrow C_{i}(A_{+}))$, where $A_{+}$ denotes the unitalization
of $A$ and the map is induced by $1_{k}\mapsto\mathbf{1}_{A_{+}}$ to the
formal unit of $A_{+}$ \cite[paragr. before Thm. 3.1]{MR997314} (this is the
definition used in \cite{MR934604} and \cite[\S 1.4.1]{MR1217970}). However,
in this text we will only ever make claims about algebras with one-sided local
units. Their bar complexes $B_{\bullet}$ are acyclic by \cite[Cor.
4.5]{MR997314} so that the obvious map $C_{\bullet}(A)\rightarrow C_{\bullet
}^{\mathsf{corr}}(A)$ is a quasi-isomorphism. As a result, for the present
text it is sufficient to take the na\"{\i}ve complex $C_{\bullet}(A)$ as the
definition, may $A$ be unital or not $-$ this is also the favourable choice
when intending to perform concrete computations.
\end{remark}

Given an algebra extension $I\hookrightarrow A\twoheadrightarrow S$, define
relative groups $C_{i}(A,I):=\ker\left(  C_{i}(A)\rightarrow C_{i}%
(A/I)\right)  $ and denote their homology by $HH_{i}(A,I)$, this is relative
Hochschild homology. Similarly $CC_{i}(A,I):=\ker\left(  CC_{i}(A)\rightarrow
CC_{i}(A/I)\right)  $ whose homology is relative cyclic homology, denoted by
$HC_{i}(A,I)$. Then the sequence in eq. \ref{lmj1} is exact, trivially by
construction. The obvious \emph{excision map}%
\begin{equation}
\rho:CC_{i}(I)\longrightarrow CC_{i}(A,I) \label{lpp2}%
\end{equation}
sending a tensor to itself is clearly well-defined. It induces the homological
excision map of eq. \ref{lpj20}, so this direction of the map in Thm.
\ref{ANNOUNCER_ThmExplicitFormula} is easy to describe explicitly. Providing
an explicit inverse is less immediate.

\section{The proof}

\subsection{\label{subsect_SimplestFilt}The Guccione-Guccione filtration}

Henceforth, we shall assume that $I$ has local left units.\ The case of local
right units would be entirely analogous. We shall also assume that $A$ has
local left units; this less natural assumption solely serves the purpose to
have the simple description of cyclic homology as in
\S \ref{marker_PreparationsAndDefOfHH} available (cf. Rmk. in
\S \ref{marker_PreparationsAndDefOfHH}). Define vector subspaces%
\[
F_{p}C_{n}(A)=\left\{
\begin{array}
[c]{l}%
k\text{-linear subspace spanned by }f_{0}\otimes\cdots\otimes f_{n}\text{
with}\\
f_{0},\ldots,f_{n-p}\in I
\end{array}
\right\}
\]
This is a filtration $F_{0}C_{\bullet}(A)\subseteq F_{1}C_{\bullet
}(A)\subseteq\ldots$, $F_{[\geq]n+1}C_{n}(A)=C_{n}(A)$ so that $\bigcup
_{p\geq0}F_{p}C_{\bullet}(A)=C_{\bullet}(A)$. This also induces a filtration
$F_{p}C_{\bullet}(A,I)$. We have $F_{0}C_{n}(A)=C_{n}(I)$.

(this filtration is of course inspired by the filtration in the the original
proof \cite{MR934604}, but by prescribing the position of the slots with
values in $I$ the computations simplify. This idea originates from
\cite{MR1370842})

Split $A$ as a $k$-vector space as $\nu:A\simeq I\oplus(A/I)$. Pick bases
$\mathcal{B}_{I}$ of $I$, $\mathcal{B}_{A/I}$ of $A/I$. Then $\mathcal{B}%
:=\mathcal{B}_{I}\cup\mathcal{B}_{A/I}$ identifies a basis of $A$ through
$\nu^{-1}$. Then equip the $F_{p}C_{n}(A)$ with bases $\{b_{0}\otimes
\cdots\otimes b_{n}\mid b_{0},\ldots,b_{n-p}\in\mathcal{B}_{I}$;
$b_{n-p+1},\ldots,b_{n}\in\mathcal{B\}}$. Call this \emph{standard tensor
basis}. All this depends on choices, either of which will be good enough. By
direct inspection:

\begin{lemma}
$F_{\bullet}$ is a filtration by subcomplexes, i.e. $bF_{p}C_{\bullet
}(A)\subseteq F_{p}C_{\bullet-1}(A)$.
\end{lemma}

For a pure tensor $\varphi=f_{0}\otimes\cdots\otimes f_{n}$ we will use the
shorthand notation $\varphi^{(\ell)}:=f_{0}\otimes\cdots\otimes f_{n-\ell}$
(the last $\ell$ slots removed) and write $\mathsf{in}(\varphi):=f_{0}$ and
$\mathsf{term}(\varphi):=f_{n}$ for the initial and terminal slot.

\begin{proposition}
\label{prop_main_descent}Suppose $\varphi\in F_{p}C_{n}(A)$ with $p\leq n$ is
a cycle (i.e. $b\varphi=0$).

\begin{enumerate}
\item Then it is homologous to a representative $\varphi^{\prime}\in
F_{p-1}C_{n}(A)$.

\item Write $\varphi=\sum\lambda_{j}\varphi_{j}$ with $\lambda_{j}\in k$ and
$\varphi_{j}$ a pure tensor in the standard tensor basis. Suppose $e\in I$ is
a local left unit for $\{\mathsf{in}(\varphi_{j})\}$. Then for each
$\varphi_{j}=f_{0}\otimes f_{1}\otimes\cdots\otimes f_{n}$ one can define%
\begin{equation}
\varphi_{j}^{\prime}:=\left(  -1\right)  ^{n+1}\left(  e\otimes\mathsf{term}%
(\varphi_{j})\varphi_{j}^{(1)}-\mathsf{term}(\varphi_{j})e\otimes\varphi
_{j}^{(1)}\right)  \label{lpwn1}%
\end{equation}
so that $\varphi^{\prime}:=\sum\lambda_{i}\varphi_{j}^{\prime}$ is an explicit solution.
\end{enumerate}
\end{proposition}

Note that we could equivalently demand $\varphi\in F_{p}C_{n}(A,I)$ and obtain
$\varphi^{\prime}\in F_{p-1}C_{n}(A,I)$ for the output.

\begin{proof}
Write $\varphi=\sum\lambda_{j}\varphi_{j}$ with $\lambda_{j}\in k$ and
$\varphi_{j}$ a pure tensor in the standard tensor basis. Since $p\leq n$ the
initial slot of each $\varphi_{j}$ lies in $I$. Let $e\in I$ be a local left
unit for the finite set $\{\mathsf{in}(\varphi_{j})\}$; exists by our
assumption on $I$. For each $\varphi_{j}=f_{0}\otimes f_{1}\otimes
\cdots\otimes f_{n}\in I\otimes A^{\otimes n}$ define $\varphi_{j}^{\prime}$
as in eq. \ref{lpwn1}; more precisely (ignoring the superscripts)
\begin{align*}
\varphi_{j}^{\prime}  &  =(-1)^{n+1}(\overset{0}{e}\otimes\overset{1}%
{f_{n}f_{0}}\otimes\overset{2}{f_{1}}\otimes\cdots\otimes\overset
{n-(p-1)}{f_{n-p}}\otimes\cdots\otimes f_{n-1}\\
&  -\overset{0}{f_{n}e}\otimes\overset{1}{f_{0}}\otimes\cdots\overset
{n-(p-1)}{f_{n-p}}\cdots\otimes f_{n-1})
\end{align*}
Next, define $\varphi^{\prime}:=\sum\lambda_{j}\varphi_{j}^{\prime}$. Firstly,
we observe $\varphi^{\prime}\in F_{p-1}C_{n}(A)$; this is clear from counting
indices (which for the comfort of the reader we have spelled out above). Next,
we need to check that $\varphi^{\prime}$ is homologous. To this end, define
for all $j$ the element $G\varphi_{j}:=e\otimes\varphi_{j}\in F_{p}C_{n+1}%
(A)$; this even lies in $F_{p-1}$. We compute $b(G\varphi_{j})$ straight from
the definition, giving%
\[
bG\varphi_{j}=\varphi_{j}-e\otimes b\varphi_{j}+(-1)^{n}\left(  e\otimes
f_{n}\varphi_{j}^{(1)}-f_{n}e\otimes\varphi_{j}^{(1)}\right)  \text{.}%
\]
To obtain this we have used the crucial fact that $e$ acts as a left unit on
the inital slots of all $\varphi_{j}$. Thus,%
\begin{align*}
&  \varphi=\sum\lambda_{j}\varphi_{j}=b(\sum\lambda_{j}G\varphi_{j})+e\otimes
b\left(  \sum\lambda_{j}\varphi_{j}\right) \\
&  \qquad+\sum\lambda_{j}(-1)^{n+1}\left(  e\otimes f_{n}\varphi_{j}%
^{(1)}-f_{n}e\otimes\varphi_{j}^{(1)}\right)  =b(\ldots)+e\otimes
b\varphi+\varphi^{\prime}%
\end{align*}
However, by assumption $\varphi$ is a cycle, i.e. $b\varphi=0$, and
$b(\ldots)$ is a boundary, so in homology we have $\varphi\equiv
\varphi^{\prime}$.
\end{proof}

\begin{corollary}
\label{Lemma_ApplyTToGetToF0}Every cycle $\varphi\in F_{n}C_{n}(A,I)$ is
homologous to a cycle $\varphi^{\prime}\in F_{0}C_{n}(A,I)$.
\end{corollary}

\begin{proof}
Just apply Prop. \ref{prop_main_descent} repeatedly $n$ times.
\end{proof}

\subsection{A refined filtration}

There is a better filtration than $F_{p}$, namely the cyclic symmetrization:
Write $tF_{p}$ for the filtration after applying $t$ (as in eq. \ref{lpp3}),
i.e. $tF_{p}C_{n}(A)=\{\varphi\mid t\cdot\varphi\in F_{p}C_{n}(A)\}$. Now
define $\tilde{F}_{p}C_{n}(A):=%
{\textstyle\sum\nolimits_{j=0}^{n}}
(t^{j}F_{n})C_{n}(A)$. Explicitly, $\tilde{F}_{p}C_{n}(A)$ is the subspace
spanned by pure tensors with $n-p+1$ \textit{cyclically successive} slots in
$I$. As for $F_{p}$, we find $\tilde{F}_{0}C_{n}(A)=C_{n}(I)$ and $\tilde
{F}_{n}C_{n}(A)$ is the subspace spanned from pure tensors with at least one
slot in $I$. The true advantage of $\tilde{F}_{p}$ is that it exhausts the
relative homology group:

\begin{lemma}
\label{marker_AllInNFilteredPartModuloCyclicPermutations}$C_{n}(A,I)=\tilde
{F}_{n}C_{n}(A,I)$.
\end{lemma}

\begin{proof}
Clearly in the standard tensor basis $C_{i}(A,I)=\ker\left(  C_{i}%
(A)\rightarrow C_{i}(A/I)\right)  $ is the subspace spanned by those
$\varphi:=f_{0}\otimes\cdots\otimes f_{n}$ with at least one $f_{j}%
\in\mathcal{B}_{I}$.
\end{proof}

Now we transport the above considerations to cyclic homology. Almost
everything goes through: The filtration $F_{p}$ does not make sense on
$CC_{\bullet}(A)$ since it is not preserved by $t$, but $\tilde{F}_{p}$ is
clearly well-defined.

\begin{lemma}
\label{Lemma_TransportToCyclic}We have

\begin{enumerate}
\item $CC_{n}(A,I)=\tilde{F}_{n}CC_{n}(A,I)$ and

\item for every cycle $\varphi\in\tilde{F}_{n}CC_{n}(A,I)$ there is a
representative in $F_{n}C_{n}(A,I)$ (under the map $F_{n}C_{n}(A,I)\rightarrow
\tilde{F}_{n}CC_{n}(A,I)$ ) and we have $\varphi\equiv\varphi^{\prime}$ with
$\varphi^{\prime}\in\tilde{F}_{0}CC_{n}(A,I)$.
\end{enumerate}
\end{lemma}

\begin{proof}
For the first claim pick a lift from $CC_{n}(A,I)$ to $C_{n}(A,I)$, then apply
Lemma \ref{marker_AllInNFilteredPartModuloCyclicPermutations}. For the second
claim, write $[\varphi]$ for an equivalence class under the cyclic
permutations $t$. Let $[\varphi]\in\tilde{F}_{n}CC_{n}(A,I)$ be given. Then
$[\varphi]=\sum\lambda_{j}[\varphi_{j}]$ with $\varphi_{j}$ pure tensors in
our standard tensor basis. For each $\varphi_{j}=[f_{0}\otimes\cdots\otimes
f_{n}]$ at least one slot $f_{i}$ lies in $I$, so we may pick the suitable
permutation $t^{i}\varphi_{j}$ so that wlog $f_{0}\in I$; giving a lift of
$\varphi_{j}$ to $F_{n}C_{n}(A,I)$; and then wlog we have a representative
$\varphi\in F_{n}C_{n}(A,I)$. Corollary \ref{Lemma_ApplyTToGetToF0} applies,
giving $\varphi\equiv\varphi^{\prime}$; this holds invariably since cyclic
homology has the same differential $b$.
\end{proof}

\begin{proof}
[Proof of Thm. \ref{ANNOUNCER_ThmExcision}]By Lemma
\ref{Lemma_TransportToCyclic} every $\varphi\in HC_{n}(A,I)$ has a
representative in the filtration step $\tilde{F}_{n}CC_{n}(A,I)$ and it
satisfies $\varphi\equiv\varphi^{\prime}$ with $\varphi^{\prime}\in\tilde
{F}_{0}CC_{n}(A,I)=CC_{n}(I)$. But this just means that $HC_{n}(A,I)=\tilde
{F}_{0}HC_{n}(A,I)=HC_{n}(I)$.
\end{proof}

It is clear that this actually yields a method to produce a concrete
representative in $HC_{n}(I)$ just by evaluating $\varphi^{\prime}$ in
concrete terms. We will do this in the next section.

\section{Proof of the explicit formula}

In this section we prove Thm. \ref{ANNOUNCER_ThmExplicitFormula} \& Thm.
\ref{ANNOUNCER_ThmExplicitFormula2}. We keep the assumptions of the last section.

\begin{proposition}
\label{marker_HHFormula}Suppose $\varphi\in F_{n}C_{n}(A,I)$ is a cycle, i.e.
$b\varphi=0$. Write $\varphi=\sum\lambda_{j}\varphi_{j}$ with each
$\varphi_{j}$ a pure tensor in our standard tensor basis, say $\varphi
_{j}=f_{0}\otimes f_{1}\otimes\cdots\otimes f_{n}$ with $f_{0}\in I$.

\begin{itemize}
\item Let $e_{n}\in I$ be a local left unit for $\bigcup\{f_{0}\}$, where the
union runs over the $f_{0}$-slots of all $\varphi_{j}$,

\item and for $i\leq n$ let $e_{i-1}$ be a local left unit for $\{e_{i}%
\}\cup\bigcup\{f_{i}e_{i}\}$, where the union runs over the $f_{i}$-slots of
all $\varphi_{j}$.
\end{itemize}

For each $\varphi_{j}=f_{0}\otimes f_{1}\otimes\cdots\otimes f_{n}$ define%
\begin{equation}
\varphi_{j}^{\prime}=\sum_{s_{1}\ldots s_{n}\in\{\pm\}}\left(  -1\right)
^{s_{1}+\cdots+s_{n}}\underset{s_{1}}{\underbrace{%
\begin{array}
[c]{c}%
e_{1}\otimes f_{1}\\
f_{1}e_{1}\otimes
\end{array}
}}\underset{s_{2}}{\underbrace{%
\begin{array}
[c]{c}%
e_{2}\otimes f_{2}\\
f_{2}e_{2}\otimes
\end{array}
}}\cdots\underset{s_{n}}{\underbrace{%
\begin{array}
[c]{c}%
e_{n}\otimes f_{n}\\
f_{n}e_{n}\otimes
\end{array}
}}f_{0} \label{lpp11}%
\end{equation}
(where for each underbrace we take the upper term if $s_{i}=+$, the lower if
$s_{i}=-$) and then $\varphi^{\prime}:=\sum\lambda_{j}\varphi_{j}^{\prime}\in
C_{n}(I)$ is an explicit representative of the same homology class as
$\varphi$.
\end{proposition}

\begin{remark}
Instead of a single $\varphi\in F_{n}C_{n}(A,I)$ we can work with finitely
many $\varphi^{1},\ldots,\varphi^{d}\in F_{n}C_{n}(A,I)$ and find a uniform
choice of the $e_{i}$ by taking the finite union of the sets $\{\mathsf{in}%
(\varphi_{j})\}$, $\{e_{i}\}\cup\bigcup\{f_{i}e_{i}\}$ appearing for the
individual $\varphi^{\alpha}$ instead.
\end{remark}

\begin{remark}
The same result holds for cyclic homology, with exactly the same proof. This
result does not prove excision for Hochschild homology since it will generally
not be true that $F_{n}HH_{n}(A,I)=HH_{n}(A,I)$.
\end{remark}

\begin{proof}
We can construct a representative of the homology class $[\varphi]$ in
$F_{0}C_{n}(A)=C_{n}(I)$ by using the procedure $\varphi\rightsquigarrow
\varphi^{\prime}$ of Prop. \ref{prop_main_descent} iteratively $n$ times. For
each iteration the element $e$ will need to be different; let us write $e_{i}$
for the element appearing in the $(n+1-i)$-th iteration $-$ we start counting
with $i:=1$; we pick and fix these $e_{n},\ldots,e_{1}\in I$. We can now
reduce the computation to pure tensors in our standard tensor basis by
linearity: Suppose $\varphi=f_{0}\otimes\cdots\otimes f_{n}$ with $f_{0}\in
I$. Recall that we have%
\begin{equation}
\varphi^{\prime}=\left(  -1\right)  ^{n+1}\left(  e\otimes\mathsf{term}%
(\varphi)\varphi^{(1)}-\mathsf{term}(\varphi)e\otimes\varphi^{(1)}\right)
\qquad\text{(for a suitable }e\text{)} \label{lpj4}%
\end{equation}
Write $t_{0}:=\varphi$ and $t_{i}:=t_{i-1}^{\prime}$ (with the prime
superscript indicating the procedure of Prop. \ref{prop_main_descent}). Note
that this construction applied to a pure tensor gives a linear combination of
two pure tensors.\ Clearly $t_{i}$ will be a linear combination of $2^{i}$
pure tensors, we will write%
\begin{equation}
t_{i}=\sum_{s_{1}\ldots s_{i}\in\{\pm\}}t_{i}^{s_{1}\ldots s_{i}}%
\text{,}\qquad\text{(for }i\geq1\text{)} \label{lpj5}%
\end{equation}
where the superscripts $s_{j}\in\{+,-\}$ encode whether we have picked the
first or second term in eq. \ref{lpj4}. To be precise: If $t_{i-1}$ (for
$i\geq1$) comes with a presentation as in eq. \ref{lpj5},%
\begin{align}
&  t_{i}=t_{i-1}^{\prime}=\sum_{s_{1}\ldots s_{i-1}\in\{\pm\}}\left(
t_{i-1}^{s_{1}\ldots s_{i-1}}\right)  ^{\prime}=\left(  -1\right)  ^{n+1}%
\sum_{s_{1}\ldots s_{i-1}\in\{\pm\}}\sum_{s_{i}\in\{\pm\}}\label{lpj6}\\
&  \qquad\ldots\left\{
\begin{array}
[c]{ll}%
e_{n+1-i}\otimes\mathsf{term}(t_{i-1}^{s_{1}\ldots s_{i-1}})\cdot\left(
t_{i-1}^{s_{1}\ldots s_{i-1}}\right)  ^{(1)} & \text{if }s_{i}=+\\
-\mathsf{term}(t_{i-1}^{s_{1}\ldots s_{i-1}})e_{n+1-i}\otimes\left(
t_{i-1}^{s_{1}\ldots s_{i-1}}\right)  ^{(1)} & \text{if }s_{i}=-
\end{array}
\right.  \text{,} \label{lpj7}%
\end{align}
where the formula in eq. \ref{lpj4} was applicable in eq. \ref{lpj7} since the
expression was decomposed into pure tensors already; $e_{n+1-i}$ denotes the
local left unit picked in this step. Eq. \ref{lpj7} gives us a presentation of
$t_{i}$ as in eq. \ref{lpj5} with%
\begin{equation}
t_{i}^{s_{1}\ldots s_{i}}=\left(  -1\right)  ^{n+1}\left\{
\begin{array}
[c]{ll}%
e_{n+1-i}\otimes\mathsf{term}(t_{i-1}^{s_{1}\ldots s_{i-1}})\cdot\left(
t_{i-1}^{s_{1}\ldots s_{i-1}}\right)  ^{(1)} & \text{if }s_{i}=+\\
-\mathsf{term}(t_{i-1}^{s_{1}\ldots s_{i-1}})e_{n+1-i}\otimes\left(
t_{i-1}^{s_{1}\ldots s_{i-1}}\right)  ^{(1)} & \text{if }s_{i}=-
\end{array}
\right.  \text{.} \label{lpj8}%
\end{equation}
We may now inductively evaluate $\mathsf{term}(t_{i-1}^{s_{1}\ldots s_{i-1}}%
)$: The above equation (just as well as eq. \ref{lpj4}) shows that
irrespective of $s_{i}$ each iteration of eq. \ref{lpj4} removes the last slot
in each pure tensor. Thus, $\mathsf{term}(t_{i}^{s_{1}\ldots s_{i}})=f_{n-i}$
for $i\geq0$. To evaluate the initial term, we follow eq. \ref{lpj7} and see
that%
\begin{equation}
\mathsf{in}(t_{i}^{s_{1}\ldots s_{i-1}s_{i}})=\left(  -1\right)
^{n+1}\left\{
\begin{array}
[c]{ll}%
e_{n-i+1} & \text{if }s_{i}=+\\
-f_{n-i+1}e_{n-i+1} & \text{if }s_{i}=-
\end{array}
\right.  \label{lpp6}%
\end{equation}
for $i\geq1$. This provides us with an inductive description of the initial
terms which we will need later. Eq. \ref{lpj8} simplifies in view of our
explicit knowledge of the terminal terms and unwinding this inductive formula
we get%
\begin{align*}
t_{n}^{s_{1}\ldots s_{n}}  &  =\left(  \left(  -1\right)  ^{n+1}\right)
^{n}\\
&  \left(  -1\right)  ^{s_{1}+\cdots+s_{n}}\underset{s_{1}}{\underbrace{%
\begin{array}
[c]{c}%
e_{1}\otimes f_{1}\\
f_{1}e_{1}\otimes
\end{array}
}}\underset{s_{2}}{\underbrace{%
\begin{array}
[c]{c}%
e_{2}\otimes f_{2}\\
f_{2}e_{2}\otimes
\end{array}
}}\cdots\underset{s_{n}}{\underbrace{%
\begin{array}
[c]{c}%
e_{n}\otimes f_{n}\\
f_{n}e_{n}\otimes
\end{array}
}}t_{0}^{(n)}\text{.}%
\end{align*}
Using eq. \ref{lpj5} and $t_{0}^{(n)}=\varphi^{(n)}=f_{0}$, we get eq.
\ref{lpwn2}. Finally, we need to identify the requirements on how the $e_{i}$
can be chosen: From the assumptions of Prop. \ref{prop_main_descent} (and the
discussion after eqs. \ref{lpj6} \& \ref{lpj7}) we see that $e_{n+1-i}$ needs
to act as a left unit on all the initial slots of the pure tensors appearing
in $t_{i-1}$. By eq. \ref{lpp6} this means that $e_{n}$ needs to act as a left
unit on the $f_{0}$, and for $i\geq1$ the element $e_{n+1-i}$ needs to act as
a left unit on $\{e_{n-i+2},f_{n-i+2}e_{n-i+2}\}$. We obtain our claim by
re-indexing: With $i^{\prime}:=n-i+1$ we get that $e_{i^{\prime}}$ needs to
act as a left unit on the initial slots of the pure tensors in $t_{n-i^{\prime
}}$, these are those with initial slots $f_{0}$ or $\{e_{i^{\prime}%
+1},f_{i^{\prime}+1}e_{i^{\prime}+1}\}$.
\end{proof}

\begin{proof}
[Proof of Thms. \ref{ANNOUNCER_ThmExplicitFormula} and
\ref{ANNOUNCER_ThmExplicitFormula2}]For cyclic homology the same argument
applies, but is stronger: By Lemma \ref{Lemma_TransportToCyclic} every
$\varphi\in HC_{n}(A,I)$ has a representative in $F_{n}C_{n}(A,I)$. As above,
we get a representative in $\tilde{F}_{0}CC_{n}(A,I)=CC_{n}(I)$. As we had
already remarked above, the procedure generalizes to finitely many elements
$\varphi^{\alpha}$ by picking the local units common for them, so we get the
result for $V$ with $V^{\prime}$ the span of the choice of representatives
$\varphi_{j}^{\alpha}$.
\end{proof}

\begin{acknowledgement}
I would like to thank V. Alekseev and M. Wodzicki for their e-mails. Moreover,
I would like to express my gratitude to the Essen Seminar for Algebraic
Geometry and Arithmetic for offering a most friendly and stimulating
scientific atmosphere.
\end{acknowledgement}

\bibliographystyle{amsalpha}
\bibliography{ollinewbib}

\end{document}